\documentclass[12pt]{article}            

\usepackage{amsmath,amssymb,latexsym, amsbsy, amsfonts, amscd, amsthm, mathrsfs, xy,comment, xcolor}
\usepackage{multirow}
\usepackage{hyperref}
\usepackage{enumitem, empheq, url}
 \usepackage[bibtex-style]{amsrefs}
\xyoption{all}
\usepackage{tikz-cd}
\usepackage{rotating, scalerel}

\theoremstyle{plain}

\newtheorem{theorem}{Theorem}[section]

\newtheorem{prop}[theorem]{Proposition}
\newtheorem{corollary}[theorem]{Corollary}

\newtheorem{lemma}[theorem]{Lemma}

\theoremstyle{definition}
\newtheorem{definition}[theorem]{Definition}

\newtheorem{example}[theorem]{Example}

\long\def\symbolfootnote[#1]#2{\begingroup
\def\thefootnote{\fnsymbol{footnote}}\footnote[#1]{#2}\endgroup}

\def\lra{\longrightarrow}

\DeclareMathOperator{\GL}{GL}

\def\PP{{\mathbf P}}

\def\cG{\mathcal{G}}

\def\cA{\mathcal{A}}

\def\1{\mf{1}}

\DeclareMathOperator{\sms}{ss}

\DeclareMathOperator{\Fitt}{Fitt}
\DeclareMathOperator{\Ann}{Ann}

\DeclareMathOperator{\Frac}{Frac}

\DeclareMathOperator{\Hom}{Hom}

\DeclareMathOperator{\tr}{tr}

\DeclareMathOperator{\im}{im}

\newcommand{\mat}[4]{\begin{pmatrix} {#1} & {#2} \\ {#3} & {#4} \end{pmatrix}}

\newcommand{\mf}{\mathfrak }

\def\fm{\mathfrak{m}}

\def\fm{\mathfrak{m}}

\def\T{\mathbf{T}}
\def\Z{\mathbf{Z}}

\def\Q{\mathbf{Q}}

\def\G{\mathbf{G}}
\def\B{\mathbf{B}}

\def\bdf{\begin{defn}}
\def\edf{\end{defn}}

\def\cO{\mathcal{O}}

\def\cA{\mathcal{A}}

\def\cG{\mathcal{G}}

\def\Gal{{\rm Gal}}

\def\ab{{\rm ab}}

\def\ab{\text{ab}}

\begin{document}
\baselineskip 15.8pt

\title{On Constructing Extensions of Residually Isomorphic Characters}
\author{Samit Dasgupta\footnote{The author is supported by NSF grant DMS--2200787.}}
\maketitle

\begin{abstract}
This is an exposition of our joint work with Kakde, Silliman, and Wang, in which we prove a version of Ribet's Lemma for $\GL_2$ in the residually indistinguishable case.  We suppose we are given a Galois representation taking values in the total ring of fractions of a complete reduced Noetherian local ring 
$\T$, such that the characteristic polynomial of the representation is reducible modulo some ideal $I \subset \T$.  We assume that the two characters that arise are congruent modulo the maximal ideal of $\T$.  We construct an associated Galois cohomology class valued in a $\T$-module that is ``large'' in the sense that its Fitting ideal is contained in $I$.  We make some simplifying assumptions that streamline the exposition---we assume the two characters are actually equal, and we ignore the local conditions needed in arithmetic applications.
\end{abstract}

\tableofcontents

\bigskip

\section{Introduction}

In 1970, Ribet proved the converse of Herbrand's Theorem \cite{ribet}.  When I learned about this in graduate school, Dick Gross recalled to me that the methods introduced by Ribet ``were like a thunderbolt'' in the number theory community at the time.  Over 50 years later, Ribet's method remains a central force in algebraic number theory, particularly in Iwasawa Theory.  Perhaps the seminal work on the topic is the beautiful book of Joel Bella\"iche and Gaetan Chenevier \cite{bc}.  
This book presents a very general form of Ribet's approach and also describes deep arithmetic applications.
An introduction written for a more general audience is given by Mazur \cite{mazur}.

The  goal of Ribet's method is to construct a nontrivial Galois cohomology class from the knowledge that an $L$-function is appropriately divisible.
Typically, we will be given a special value of an $L$-function, which we denote by $L$, lying in a ring $\T$ (e.g.\ $\T = \Z_p$ or $\T = \Z_p[[T]]$), and we assume assume that $L$ lies in some ideal $I \subset \T$.
The value $L$ will be associated to two (or more) representations of the Galois group of a number field $F$ over $R$, say $\rho_1$ and $\rho_2$.  One then wants to construct a nontrivial class in 
\[ H^1(G_F, \overline{\rho}_1 \otimes \overline{\rho}_2^*), \]
where $\rho_2^*$ is the dual of $\rho_2$ and $\overline{\rho}_i$ denotes the reduction of ${\rho}_i$ modulo $I$.  In Ribet's original setting, he had $\T = \Z_p$, $I = (p)$, and $\rho_1, \rho_2$ one dimensional characters of $G_\Q$.  Specifically, $\rho_1$ was the trivial character and $\rho_2$ a nontrivial character of $\Gal(\Q(\mu_p)/\Q)$.

To produce  an extension, Ribet constructed a cusp form congruent to the Eisenstein series associated to $\rho_1$ and $\rho_2$.  In his case, one can show that the cusp form is an eigenform. One therefore obtains a representation \[ \rho\colon G_\Q \lra \GL_2(K), \]
where $K$ is a finite extension of $\Q_p$, such that 
\begin{equation} \label{e:rhodef}
 \tr(\rho) \equiv \rho_1 + \rho_2, \qquad \det(\rho) \equiv \rho_1\rho_2 \pmod{I}. \end{equation}
The representation $\rho$ can be conjugated to land in $\GL_2(\cO_K)$,
and the
Brauer--Nesbitt Theorem implies that one can choose a basis so that the reduction of $\rho$ modulo the maximal ideal $\fm_K \subset \cO_K$ has the form
\[ \overline{\rho} = \mat{\overline{\rho}_1}{*}{0}{\overline{\rho}_2}. \]
Ribet then proves an important lemma which shows that the basis can be chosen so that the image of $*$ in $\cO_K/\fm_K$ defines a non-trivial class in $H^1(G_F, \overline{\rho}_1 \otimes \overline{\rho}_2^{-1})$.  Furthermore, Ribet proves certain local conditions satisfied by this non-trivial class.  It is elementary class field theory to prove that the existence of this class implies the converse to Herbrand's Theorem.
We state Ribet's Lemma formally as follows.

\begin{theorem}[Ribet's Lemma, Version 1] \label{t:v1}
Let $\T$ be
 a complete DVR,
and let $\fm$ denote its maximal ideal.
Let $\cG$ be a compact group.
Suppose we are given a continuous irreducible 
representation
\[ \rho\colon \cG \longrightarrow \GL_2(K), \qquad K = \Frac(\T), \]
such that \begin{equation} \label{e:rhocong}
 \text{char}(\rho(g)) \equiv (x - \chi_1
(g))(x - \chi_2(g)) \pmod{\mathfrak{m}} 
\end{equation}
for two characters $\chi_1, \chi_2 \colon \cG \lra \T^*.$
Then there exists a non-zero cohomology class 
\[ \kappa \in H^1(\cG, \mathbf{\T/\mathfrak{m}}(\chi_2^{-1}\chi_1)). \]
\end{theorem}

In more general settings, such as that employed by Mazur--Wiles \cite{mw} and Wiles \cite{wiles} in their study of the Iwasawa Main Conjecture, the ring $\T$ cannot be assumed to be a DVR.  It will usually be a complete local Noetherian ring, perhaps reduced.  An example of such a ring that is not a DVR is 
\[ \T = \{(a, b) \in \Z_p \times \Z_p \colon a \equiv b \pmod{p} \}. \]
This example corresponds to two eigenforms with Hecke eigenvalues in $\Z_p$ that are congruent to each other modulo $p$. The total ring of fractions of $\T$ is $K = \Frac(\T) = \Q_p \times \Q_p$.

 In addition, the ideal $I \subset \T$
modulo which the characteristic polynomial of $\rho$ factors will in general not be the maximal ideal of $\T$.
One then wants to construct a cohomology class that generates a module that (in a sense we will make precise in a moment) is ``as large as" $\T/I$.
The first version of Ribet's Lemma that applies in this case was proven by Mazur--Wiles  and Wiles.  Their work was groundbreaking and had a profound impact, leading to Wiles' theory of pseudorepresentations.
In the statement below, a cohomology class $\kappa$ valued in a $\T$-module $M$ is called {\em surjective} if the image of every representative cocycle generates $M$ as a $\T$-module. 
\begin{theorem}[Ribet's Lemma, Version 2] \label{t:v2}
Let $\T$ be
 a complete reduced local Noetherian ring,
and let $\fm$ denote its maximal ideal.
Let $\cG$ be a compact group.
Suppose we are given a continuous
representation
\[ \rho\colon \cG \longrightarrow \GL_2(K), \qquad K = \Frac(\T), \]
such that for any projection onto a field $K \rightarrow k$, the projection of $\rho$ to a representation $\cG \rightarrow \GL_2(k)$ is irreducible.
Let $I \subset \T$ be a proper ideal. Suppose that \begin{equation} \label{e:rhocong2}
 \text{char}(\rho(g)) \equiv (x - \chi_1
(g))(x - \chi_2(g)) \pmod{I} \end{equation}
for two characters $\chi_1, \chi_2 \colon \cG \lra \T^*$ satisfying $\chi_1 \not\equiv 
\chi_2 \pmod{\fm}$.
Then there exists a fractional ideal of $\T$, $B \subset K$, and a surjective cohomology class
\[ \kappa \in H^1(G, B/IB(\chi_2^{-1}\chi_1)). \]
\end{theorem}

The assumption $\chi_1 \not\equiv \chi_2 \pmod{\fm}$ is essential in Wiles' approach to Theorem~\ref{t:v2}, and it has important consequences.
In the Main Conjecture of Iwasawa Theory, one has $\chi_1$ equal to the trivial character and $\chi_2$ equal to a totally odd character of a totally real field.
Let $c$ denote complex conjugation.  When $p \neq 2$, we have $\chi_1(c) = 1 \not\equiv -1 = \chi_2(c) \pmod{p}$.  However when $p=2$ we may have $\chi_1 \equiv \chi_2 \pmod{p}$, and Theorem~\ref{t:v2} cannot be applied.  This is the main reason that Wiles sets $p \neq 2$ in his proof of the Main Conjecture.

The purpose of this exposition is to describe the main theorem of our paper \cite{jjm}, joint with Mahesh Kakde, Jesse Silliman, and Jiuya Wang, in which we establish a version of Ribet's Lemma that holds even if $\chi_1 \equiv \chi_2 \pmod{\fm}$.  Here we describe the proof of the following result, a simplified form of Theorem 2.1 of {\em loc.\ cit.}
\begin{theorem}[Ribet's Lemma, Version 3] \label{t:v3}
Let $\T$ be
 a complete reduced local Noetherian ring,
and let $\fm$ denote its maximal ideal.
Let $\cG$ be a compact group.
Suppose we are given a continuous
representation
\[ \rho\colon \cG \longrightarrow \GL_2(K), \qquad K = \Frac(\T), \]
such that for any projection onto a field $K \rightarrow k$, the projection of $\rho$ to a representation $\cG \rightarrow \GL_2(k)$ is irreducible.
Let $I \subset \T$ be a proper ideal. Suppose that \begin{equation} \label{e:charcong3}
 \text{char}(\rho(g)) \equiv (x - \chi_1
(g))(x - \chi_2(g)) \pmod{I} \end{equation}
for two characters $\chi_1, \chi_2 \colon \cG \lra \T^*$. Then there exists a finitely generated $\T$-module $M$ and a surjective cohomology class
\[ \kappa \in H^1(G, M(\chi_2^{-1}\chi_1)) \]
such that \[ \Fitt_{\T}(M) \subset I. \]
\end{theorem}

  The ideal $\Fitt_{\T}(M)$ is the 0th Fitting ideal of the module $M$, which will be defined in \S\ref{s:fitt}.  Intuitively, the inclusion $\Fitt_{\T}(M) \subset I$ says that $M$ is ``large.''

In \cite{bsapet}, we prove the Brumer--Stark conjecture at $p=2$ using a suitably generalized version of Theorem~\ref{t:v3}.  Previously, we proved the conjecture over $\Z[1/2]$ in \cite{bs}, with the prime $p=2$ being avoided for reasons of residual distinguishability.   We hope that further strengthenings of Theorem~\ref{t:v3} (for example, to groups other than $\GL_2$) could have other arithmetic applications.  

In this paper, we simplify notation by setting $\chi_1$ and $\chi_2$ to be the trivial character; the case of general $\chi_1, \chi_2$ requires no extra ideas, but the notation is heavier.  A more significant change in this paper relative to \cite{jjm} is that all the versions of Ribet's method stated above do not include local conditions on the cohomology classes constructed.  In  arithmetic applications, ranging from Ribet's original proof of the converse to Herbrand's theorem to our proof of the Brumer--Stark conjecture, local conditions are always necessary.  To prove the local properties we need in the Brumer--Stark context, the argument presented here is  generalized in \cite{jjm} using the Buchsbaum--Rim resolution of determinantal ideals; in this paper, the simple Kozsul complex suffices.  In {\em loc.\ cit.}~we also give a mild generalization to certain non-reduced Hecke algebras $\T$.

\bigskip

It is an honor to contribute this article to the memorial volume for Jo\"el Bella\"iche.   Jo\"el was a wonderful collaborator and dear friend.  We discussed the residually indistinguishable case of Ribet's Lemma in 2010, at which time both of us felt the problem was intractable.  It is a great sadness that I cannot share this result with my colleague who perhaps would have appreciated it the most.

\section{The DVR case} \label{s:dvr}

In this section we prove Theorem~\ref{t:v1}, Ribet's original setting.  Let $\T$ be
 a complete DVR and let $\fm$ denote its maximal ideal.
Let $\cG$ be a compact group.
We are given a continuous irreducible 
representation
\[
\rho\colon \cG \longrightarrow \GL_2(K), \qquad K = \Frac(\T), 
\]
such that  \begin{equation} \label{e:rhobardvr} \text{char}(\rho(g)) \equiv (x - \chi_1
(g))(x - \chi_2(g)) \pmod{\mathfrak{m}} \end{equation}
for two characters $\chi_1, \chi_2 \colon \cG \lra \T^*.$

\begin{lemma}  The representation $\rho$ may be conjugated to have image contained in $\GL_2(\T)$, and such that the reduction $\overline{\rho}$ has the shape
\[ \overline{\rho} = \mat{\overline{\chi}_1}{*}{0}{\overline{\chi}_2}. \]
\end{lemma}

\begin{proof}
The maximal compact subgroups of $\GL_2(K)$ are precisely the conjugates of $\GL_2(\T)$.  Since $\cG$ is compact and $\rho$ is continuous, it follows $\rho$ that may be conjugated to have image contained in $\GL_2(\T)$.  Now (\ref{e:rhobardvr}) states that
\begin{equation} \label{e:rhobardvr2}
 \text{char}(\overline{\rho}(g)) = (x - \overline{\chi}_1(g))(x - \overline{\chi}_2(g)) \text{ in } \T/\fm. \end{equation}
The Brauer--Nesbitt Theorem \cite{bn} states that if two representations over a field have the same characteristic polynomial, then their semisimplifications are isomorphic.  Therefore, (\ref{e:rhobardvr2}) implies that $\overline{\rho}^{\sms} \cong  \overline{\chi}_1 \oplus\overline{\chi}_2$.
Hence we have either
\begin{equation} \label{e:twocases}
  \overline{\rho} \cong \mat{\overline{\chi}_1}{*}{0}{\overline{\chi}_2} \qquad \text{ or }  \qquad \overline{\rho} \cong \mat{\overline{\chi}_2}{*}{0}{\overline{\chi}_1}.
\end{equation}
It is a pleasant exercise to prove (\ref{e:twocases}) directly from (\ref{e:rhobardvr2}) without reference to the full strength of the Brauer--Nesbitt Theorem.  Now, suppose we are in the second case.  Then we can conjugate $\rho$ by the matrix $\mat{\pi}{0}{0}{1}$, where $\pi$ is a uniformizer of $\T$, to obtain
\[ \overline{\rho} \cong \mat{\overline{\chi}_2}{0}{*}{\overline{\chi}_1} \cong  \mat{\overline{\chi}_1}{*}{0}{\overline{\chi}_2} \]
as desired.
\end{proof}

\begin{lemma} \label{l:kappadef}  Suppose that the representation $\rho$ has been conjugated so that its image lands in $\GL_2(\T)$ and
\[ \overline{\rho} =  \mat{\overline{\chi}_1}{\overline{b}}{0}{\overline{\chi}_2}.
\]
Then the function $\kappa(\sigma) = \overline{\chi}_2^{-1}(\sigma) \overline{b}(\sigma)$ is a 1-cocycle in $Z^1(\cG, \T/\fm(\chi_2^{-1}\chi_1))$.
\end{lemma}

\begin{proof}
Since $\overline{\rho}$ is a matrix representation, we have
\[ \overline{b}(\sigma \tau) = \overline{\chi_1}(\sigma) \overline{b}(\tau) + \overline{b}(\sigma) \overline{\chi_2}(\tau). \]
Multiplying by $\overline{\chi}_2^{-1}(\sigma \tau)$ gives the desired 1-cocycle formula for 
$\kappa = \overline{\chi}_2^{-1} \overline{b}$.
\end{proof}

What remains to prove Theorem~\ref{t:v1} is to show that after conjugating $\rho$ further, we can arrange for the cohomology class represented by $\kappa$ to be non-trivial.

\begin{proof}[Proof of Theorem~\ref{t:v1}]
In his 2008 lectures from the Clay Summer School in Hawaii, Bella\"iche gives a beautiful and conceptual proof of the fact that $\rho$ can be chosen so that the cohomology class represented by $\kappa$ is non-trivial \cite{clay}*{Proposition 1.4}.  He attributes this proof to Serre.  Here, we take the more computational approach applied by Ribet.

Since we will be applying a recursive process, let $\rho_1 = \rho$ and write
\[ \rho_1(\sigma) = \mat{a_1(\sigma)}{b_1(\sigma)}{c_1(\sigma)}{d_1(\sigma} \in \GL_2(\T). \]
If $\pi$ denotes a uniformizer of $K = \Frac(\T)$, then we have
\begin{alignat}{2}
a_1(\sigma) & \equiv \chi_1(\sigma) && \pmod{\pi}, \label{e:acong} \\
c_1(\sigma)  & \equiv 0 && \pmod{\pi}, \\
d_1(\sigma) & \equiv \chi_2(\sigma) && \pmod{\pi}
\end{alignat}
 for all $\sigma \in \cG$.
Denote the cocycle constructed in Lemma~\ref{l:kappadef} by $\kappa_1$.  

Suppose that $\kappa_1$ represents a trivial cohomology class; then there exists $x_1 \in \T$ such that
\[ \kappa_1(\sigma) \equiv (\chi_2^{-1}\chi_1(\sigma) - 1) x_1 \pmod{\pi} \]
for all $\sigma \in \cG$, or equivalently, 
\begin{equation} \label{e:bcong}
 b_1(\sigma) \equiv (\chi_1(\sigma) - \chi_2(\sigma)) x_1 \pmod{\pi}. \end{equation}
Conjugating the representation $\rho_1$, we define
\[ {\rho}_2(\sigma) = \mat{a_2(\sigma)}{b_2(\sigma)}{c_2(\sigma)}{d_2(\sigma)} = \mat{1}{x_1}{0}{\pi} \rho_1(\sigma) \mat{1}{x_1}{0}{\pi}^{-1}. \]
Using the congruences (\ref{e:acong})--(\ref{e:bcong}), we find that $\rho_2(\sigma) \in \GL_2(\T)$ and furthermore that
\begin{alignat*}{2}
a_2(\sigma) & \equiv \chi_1(\sigma) && \pmod{\pi}, \\
c_2(\sigma)  & \equiv 0 && \pmod{\pi^2}, \\
d_2(\sigma) & \equiv \chi_2(\sigma) && \pmod{\pi}.
\end{alignat*}
We are therefore once again in the setting of Lemma~\ref{l:kappadef} and obtain a cocycle
\[ \kappa_2(\sigma) = \overline{\chi}_2^{-1}(\sigma) \overline{b_2}(\sigma) \in Z^1(\cG, \T/\fm(\chi_2^{-1}\chi_1)). \]
If $\kappa_2$ represents a nontrivial cohomology class, we are done.  If not, we may repeat this process and obtain another
representation $\rho_3$, where now $\pi^3 \mid c_3(\sigma)$.  This process continues.

If at any stage, we obtain a cocycle $\kappa_i$ that represents a non-trivial class, then we are done.  If the process continues forever, then one checks that by defining $x = x_1 + \pi x_2 + \pi^2 x_3 + \cdots$, conjugating the original representation $\rho_1$ by $\mat{1}{x}{0}{1}$ leaves a representation over $K$ with a 0 in the upper right-hand corner.  This contradicts the irreducibility of $\rho_1$. 
\end{proof}

\section{The Residually Distinguishable Case}

In this section, we prove Theorem~\ref{t:v2}.  Our complete local ring $\T$ is no longer assumed to be a DVR, but we grant ourselves the assumption
$\chi_1 \not\equiv \chi_2 \pmod{\fm}$.  Fix $\tau \in \cG$ such that $\chi_1(\tau) \not\equiv \chi_2(\tau) \pmod{\fm}$.  By Hensel's Lemma, the congruence (\ref{e:rhocong2}) implies that $\rho(\tau)$ has two distinct eigenvalues $\lambda_1, \lambda_2 \in \T$ such that \begin{equation} \label{e:lambdacong}
\lambda_i \equiv \chi_i(\tau) \pmod{I} \end{equation} for $i =1,2$.
We choose a basis for $\rho$ over $K = \Frac(\T)$ such that 
\[ \rho(\tau) = \mat{\lambda_1}{0}{0}{\lambda_2}. \]  For $\sigma \in \cG$, write
\[ \rho(\sigma) = \mat{a(\sigma)}{b(\sigma)}{c(\sigma)}{d(\sigma)}. \]
Note that unlike the DVR case, we cannot ensure that $\rho$ takes values in $\GL_2(\T)$, only $\GL_2(K)$.
Nevertheless, we have the following.

\begin{lemma} \label{l:adcong}
We have $a(\sigma), d(\sigma) \in \T$ for all $\sigma \in \cG$, and furthermore
\begin{align*}
 a(\sigma) & \equiv \chi_1(\sigma) \pmod{I},  \\
 d(\sigma) & \equiv \chi_2(\sigma) \pmod{I}.
 \end{align*}
\end{lemma}
\begin{proof}
The congruence (\ref{e:rhocong2}) implies that
\begin{equation} \label{e:ad1}
 a(\sigma) + d(\sigma) \equiv \chi_1(\sigma) + \chi_2(\sigma) \pmod{I}.
 \end{equation}
 Applying this with $\sigma$ replaced by $\sigma \tau$, and noting that $a(\sigma\tau) = a(\sigma)\lambda_1$ and $d(\sigma\tau) = d(\sigma) \lambda_2$ by our choice of basis,
 we obtain \begin{equation} \label{e:ad2}
 a(\sigma)\lambda_1 + d(\sigma)\lambda_2 \equiv \chi_1(\sigma)\chi_1(\tau) + \chi_2(\sigma)\chi_2(\tau). \pmod{I} \end{equation}
 Solving the congruences (\ref{e:ad1}) and (\ref{e:ad2}) using (\ref{e:lambdacong}), we find that $a(\sigma), d(\sigma) \in \T$ and furthermore 
$ a(\sigma) \equiv \chi_1(\sigma) \pmod{I}$ and 
$d(\sigma) \equiv \chi_2(\sigma) \pmod{I}$ as desired.  
Note that this uses the fact that $\chi_1(\tau) \not \equiv \chi_2(\tau) \pmod{\fm}$ in a crucial way.
\end{proof}

We now let $B$ denote the $\T$-submodule of $K$ spanned by the $b(\sigma)$ for $\sigma \in \cG$.
Note that the irreducibility assumption on $\rho$ implies that $B$ is a fractional ideal of $\T$.  Indeed, if we write $K$ as a product of fields, then the projection of 
$\rho$ on to each field factor is irreducible, hence the projection of $B$ onto each field factor is nonzero.

In view of the congruences of Lemma~\ref{l:adcong}, the equation
\[ b(\sigma \sigma') = a(\sigma) b(\sigma') + b(\sigma) d(\sigma') \]
yields
\[ \overline{b}(\sigma \sigma') \equiv \chi_1(\sigma) \overline{b}(\sigma') + \chi_2(\sigma') \overline{b}(\sigma) \quad \text{ in } B/IB. \]
As in Lemma~\ref{l:kappadef}, multiplying by $\chi_2^{-1}(\sigma \sigma')$ shows that 
\[ \kappa(\sigma) = \chi_2^{-1}(\sigma) \overline{b}(\sigma) \in B/IB \]
defines an element of $Z^1(\cG, B/IB(\chi_2^{-1}\chi_1)).$

It remains to prove that the cohomology class represented by $\kappa$ is surjective.
Therefore let
\begin{equation} \label{e:kappa'}
 \kappa'(\sigma) = \kappa(\sigma) + (\chi_2^{-1}\chi_1(\sigma) - 1)x \end{equation}
 be a cohomologous cocycle, where $x \in B/IB$.
Let $B'$ denote the $\T$-submodule of $B/IB$ generated by the values of $\kappa'$.
  Applying (\ref{e:kappa'}) to $\tau$ and recalling that $\kappa(\tau) = 0$, we see that
\[ \kappa'(\tau) = (\chi_2^{-1}\chi_1(\tau) - 1)x \in B'. \]
Since the item in parentheses in a unit in $\T$ by assumption, we find that $x \in B'$.  Going back to (\ref{e:kappa'}), it follows that
$\kappa(\sigma) \in B'$ for all $\sigma \in \cG$.  But it is clear from the definitions of $B$ and $\kappa$ that $\kappa$ generates $B/IB$, whence $B' = B/IB$ as desired.  This completes the proof of Theorem~\ref{t:v2}.

\section{The Residually Indistinguishable Case}

In the remainder of the paper we describe the proof of Theorem~\ref{t:v3}.  Before explaining the definition of the Fitting ideal that appears in the statement of the theorem, we remark that the methods of the previous sections are not directly applicable---in the DVR case, manipulations with the uniformizer $\pi$ were essential, and in the residually distinguishable case, the  basis for $\rho$ afforded by the special element $\tau$ was critical. Indeed, in the residually indistinguishable case we do not know how to show that $a(\sigma), d(\sigma) \in \T$ (let alone that the congruences of Lemma~\ref{l:adcong} hold) in {\em any} basis.  A separate construction and proof is required.

As mentioned in the introduction, we will assume for the rest of the paper that $\chi_1$ and $\chi_2$ are trivial; this offers notational simplifications, but no significant changes to the argument.

\subsection{Fitting Ideal} \label{s:fitt}

Let $R$ be a commutative ring.  Let $M$ be a finitely presented $R$-module.  This  means that there is a short exact sequence
\[  \begin{tikzcd}
R^n \ar[r, "f"] & R^m \ar[r] & M \ar[r] & 0.
\end{tikzcd}
 \]
 \begin{definition}
The 0th Fitting ideal of $M$ over $R$, which we denote $\Fitt_R(M)$, is the ideal of $R$ generated by all $m \times m$ minors of the matrix representing the linear map $f$. By convention, if $n < m$, then $\Fitt_R(M) = 0$.
 
 \end{definition}

 We leave the proof of the following basic facts about Fitting ideals as an exercise for the reader (alternatively, she may consult \cite{nickel}).
 
 \begin{prop} We have the following.
 \begin{itemize}
 \item The Fitting ideal $\Fitt_R(M)$ depends only on $M$ up to $R$-module isomorphism, and not on the particular presentation taken.
\item If $I \subset R$ is a finitely generated ideal, then $\Fitt_R(R/I) = I$.
\item The Fitting ideal of $M$ is contained in the annihilator of $M$: $\Fitt_R(M) \subset \Ann_R(M)$.
\item If $R= \Z$ and $M$ is a finitely generated abelian group, then $\Fitt_{\Z}(M) = 0$ if $M$ is infinite and $\Fitt_{\Z}(M) = (\#M)$, the ideal generated by the size of $M$, if $M$ is finite. 
\item If $M \twoheadrightarrow M'$ is a surjection of finitely presented $R$-modules, then $\Fitt_R(M') \supset \Fitt_R(M)$.
\item (Base Change) If $S$ is an $R$-algebra and $M$ is a finitely presented $R$-module, then $\Fitt_S(M \otimes_R S) = \Fitt_R(M) \cdot S.$
 \end{itemize}
 \end{prop}
 
 \begin{corollary} Let $\T$ and $B$ be as in Theorem~\ref{t:v2}.  Then $\Fitt_{\T}(B/IB) \subset I$.
 \end{corollary}
 
 \begin{proof}  Since $B$ is a fractional ideal of $\T$, it is a faithful $\T$-module, i.e.\ $\Ann_{\T}(B) = 0$.  It follows that $\Fitt_{\T}(B) = 0$.  Therefore $\Fitt_{\T/I}(B/IB) = 0$, whence $\Fitt_{\T}(B/IB) \subset I$.
 \end{proof}
 
 Motivated by this corollary, one of the insights of Theorem~\ref{t:v3} is its statement---in the residually indistinguishable case, instead of attempting to necessarily construct a $\T$-module of the form $B/IB$ for a faithful $\T$-module $B$, one should just attempt to construct some module $M$ such that $\Fitt_{\T}(M) \subset I$.
 
 \subsection{Construction of \texorpdfstring{$M$}{M}}
 
 To prove Theorem~\ref{t:v3}, we must construct, under the assumptions of the theorem, a finitely generated $\T$-module $M$ such that $\Fitt_{\T}(M) \subset I$ and a surjective cohomology class $\kappa \in H^1(\cG, M)$.  Since we have specialized to $\chi_1 = \chi_2 = 1$, the $\cG$-action on $M$ is trivial, whence $ H^1(\cG, M)$ is the group of continuous homomorphisms $\cG \rightarrow M$.  Since $M$ is abelian, such a homomorphism necessarily factors through the abelianization $\cG^{\ab}$ and thereby induces a $\T$-module map \begin{equation} \label{e:gtm}
 \cG^{\ab} \otimes \T \lra M. \end{equation} 
 The surjectivity of $\kappa$ is simply the statement that the $\T$-module homomorphism (\ref{e:gtm}) is surjective.
 
 These considerations lead to a very natural construction of the module $M$.  Let $\Delta$ denote the augmentation ideal of $\cG$ over $\T$, i.e.\ the kernel of the $\T$-algebra homomorphism
 \begin{equation} \label{e:augdef}
  \T[\cG] \lra \T, \qquad \sum a_g[g] \mapsto \sum a_g. \end{equation}
 It is then well-known that we have a $\T$-module isomorphism
 \[ \Delta/\Delta^2 \cong \cG^{\ab} \otimes \T, \qquad \sum a_g[g] \mapsto \sum g \otimes a_g. \]
 In order to apply the assumptions of Theorem~\ref{t:v3}, we must invoke the continuous group representation $\rho \colon \cG \rightarrow \GL_2(K)$.  Note that $\rho$
 can be extended to a $\T$-algebra homomorphism (also denoted $\rho$) 
 \[ \T[\cG] \lra M_2(K). \] 
 We then define \[ M = \rho(\Delta)/\rho(\Delta^2), \]
 so there is a canonical surjection
 \[ \begin{tikzcd} \cG^{\ab} \otimes \T \cong \Delta/\Delta^2 \ar[r, "\rho", two heads] & M.  \end{tikzcd} \]
 As explained above, this homomorphism represents a surjective cohomology class \[ \kappa \in H^1(\cG, M). \]  It remains to prove that $\Fitt_{\T}(M) \subset I$, and this will take up the remainder of the paper.
 
 \subsection{Explication of Fitting Ideal} \label{s:efitt}
 
 Since $\cG$ is compact and $\rho$ is continuous, the image $\rho(\cG)$ is a compact subset of $M_2(K)$, and hence $\fm$-adically bounded.  The same is therefore true of $\rho(\T[\cG])$ and $\rho(\Delta)$.  It follows that $\rho(\Delta)$ is a finitely generated $\T$-module, and hence that $M = \rho(\Delta)/\rho(\Delta^2)$ is finitely generated as well.
 
 Let $\rho_1, \dots, \rho_r$ denote a set of $\T$-module generators for $\rho(\Delta)$, where $\rho_i = \rho(g_i - 1)$ for elements $g_i \in \cG$.
 The images of the $\rho_i$ in $M$ are $\T$-module generators, and there are two types of relations for these generators in $M$:
\begin{itemize}
\item We may have relations 
\begin{equation} \sum_{i=1}^{r} \epsilon_i \rho_i =0 \text{ in } M_2(K) \label{e:epsdef}
\end{equation} 
for constants $\epsilon_i \in \T$.  We call these relations of $\epsilon$-type.
\item For each pair $1 \le i, j \le r$, we may write 
\begin{equation}
 \rho_i \rho_j = \sum_{k=1}^{r} \delta_{ijk} \rho_k \label{e:deltadef}
 \end{equation}
for constants $\delta_{ijk} \in \T$.  These expressions vanish in $M$; we call these relations of $\delta$-type.
\end{itemize}

The Fitting ideal of $M$ is the ideal generated by all determinants $\det(D)$ where $D$ is an $r \times r$ matrix whose rows have the form 
$(\epsilon_1, \dotsc, \epsilon_r)$ for relations of $\epsilon$-type  or $(\delta_{ij1}, \delta_{ij2}, \dots, \delta_{ijr})$ for relations of $\delta$-type.
We need to show that $\det(D) \in I$ for each such matrix $D$.

\subsection{Traces and Determinants}

In order to prove that $\det(D) \in I$, we need to generate expressions involving the $\rho_i$ that are known to lie in $I$.  This is given by the following.

\begin{lemma}  For each $A \in \rho(\Delta)$, we have $\tr(A) \in I$ and $\det(A) \in I$. \label{l:tdi}
\end{lemma}

\begin{proof}
Let $\chi \colon \T[\cG] \lra \T$ denote the augmentation map defined in (\ref{e:augdef}), i.e.\ the trivial character of $\cG$ extended to a $\T$-algebra homomorphism on $\T[\cG]$. The congruence (\ref{e:charcong3}) implies that 
\[ \tr(\rho(g)) \equiv 2 \chi(g) \pmod{I}, \qquad \det(\rho(g)) \equiv \chi(g) \pmod{I} \]
for $g \in \cG$. A simple argument shows that these congruences extend to all $g \in \T[\cG]$ (see \cite{ow}*{Lemma 3.1}).
By definition, $\chi(g) = 0$ for $g \in \Delta$.  The result follows.
\end{proof}

\subsection{An Altered Matrix}

Write \[ \rho_i = \mat{a_i}{b_i}{c_i}{d_i}. \]
Given an $r \times r$ matrix $D$ of $\epsilon$-type and $\delta$-type relations as in \S\ref{s:efitt}, we define an associated matrix $D'$ obtained by altering the rows of $D$ as follows. 

\begin{itemize}
\item No change for $\epsilon$-type rows.
\item For rows of $\delta$-type, replace $\delta_{ijj}$ by $\delta_{ijj} - a_i$, replace $\delta_{iji}$ by $\delta_{iji} - d_j$, and leave the other $\delta_{ijk}$ unchanged.\end{itemize}

\begin{lemma} We have $\det(D') = 0.$ \label{l:detzero}
\end{lemma}

\begin{proof} 
Let 
\[ w = (b_1, b_2, \dots, b_r)^T  \qquad \text{ (column vector). } \]
We claim that $D' w = 0$.  For a row of $\epsilon$-type, the corresponding component of $D'w$ is $\sum_{i=1}^r \epsilon_i b_i$.  This is the upper right coefficient of the matrix  $\sum_{i=1}^r \epsilon_i \rho_i$, and hence vanishes by the definition of the $\epsilon_i$ (see (\ref{e:epsdef})).  Similarly for a row of $\delta$-type, the corresponding component of $D'w$ is \begin{equation}
 \left(\sum_{k=1}^r \delta_{ijk} b_k\right) - (a_i b_j + b_i d_j) = 0. \label{e:deltazero} \end{equation}
The quantity (\ref{e:deltazero}) vanishes because each expression in parentheses equals the upper right coefficient of $\rho_i \rho_j$; this follows on the left by the definition (\ref{e:deltadef}) of the $\delta_{ijk}$, and on the right by the definition of matrix multiplication.

The ring $K = \Frac(\T)$ is a product of fields.  To prove $\det(D') = 0$, it suffices to prove this on each field factor of $K$.  Now, in each field factor of $K$, the projection of some $b_i$ must be nonzero; otherwise the projection of $\rho$ to that field factor would be lower triangular and hence reducible, contrary to assumption.  Therefore the equation $D'w = 0$ implies $\det(D') = 0$ as desired.
\end{proof}

Our goal is to show that $\det(D) \in I$, and we have shown that $\det(D') = 0$.  It therefore suffices to show that 
$\det(D') - \det(D) \in I$.  In other words, the alterations used to go from $D$ to $D'$ are small enough to leave the determinant unchanged modulo $I$. Let us motivate our strategy to prove this with an example.

\subsection{An Example} \label{s:example}

Suppose $r=2$.   We consider a matrix with only rows of $\delta$-type, namely

\[  D = \mat{\delta_{121}}{\delta_{122}}{\delta_{211}}{\delta_{212}}, \qquad \text{ whence } \qquad D' = \mat{\delta_{121} - d_2}{\delta_{122} - a_1}{\delta_{211} - a_2}{\delta_{212} - d_1}. \]
As shorthand, write $t_i = a_i + d_i$ for the trace of $\rho_i$.  We write $d_{12} = c_1b_2 + d_1d_2$ for the ``$d$''-component of $\rho_1\rho_2$, and $t_{12} = a_{12} + d_{12}$ for the trace of $\rho_1 \rho_2$.  By multilinearity of the determinant, we have
\begin{equation} \label{e:dpd}
\det(D') - \det(D) = - \det\mat{d_2}{a_1}{\delta_{211}}{\delta_{212}} - \det \mat{\delta_{121}}{\delta_{122}}{a_2}{d_1}
+ \det \mat{d_2}{a_1 }{a_2}{d_1}. 
\end{equation}
We evaluate the determinants on the right using the substitution $t_i = a_i + d_i$ where convenient.
Then \begin{align}
\det\mat{d_2}{a_1}{\delta_{211}}{\delta_{212}} & = - t_1  \delta_{211} + (d_1 \delta_{211} + d_2 \delta_{212})    \nonumber \\
& = - t_1  \delta_{211} + d_{21} \label{e:usedd} \\
& = - t_1 \delta_{211} + (c_2b_1 + d_1d_2). \label{e:det1}
\end{align}
Note that equation (\ref{e:usedd}) uses the definition of the $\delta_{ijk}$.
Similarly
\begin{align}
\det \mat{\delta_{121}}{\delta_{122}}{a_2}{d_1}  &  = - t_2  \delta_{122} + (d_1 \delta_{121} + d_2 \delta_{122} )    \nonumber\\
& = -t_2 \delta_{122} + d_{12} \nonumber \\
&= - t_2 \delta_{122} + (c_1b_2 + d_1d_2). \label{e:det2}
\end{align}
Combining (\ref{e:dpd}), (\ref{e:det1}), and (\ref{e:det2}), we obtain

\begin{align}
\begin{split}
\det(D') - \det(D) &= t_1 \delta_{211} - (c_2b_1 + d_2d_1)  +    \\
& \ \  \ \  + t_2  \delta_{122} - (c_1b_2 + d_1 d_2 )  +   \\
&  \ \ \ \   +  (d_1d_2 - a_1 a_2)    \nonumber
\end{split} \\
&= t_1 \delta_{211} + t_2  \delta_{122} - t_{12}. \label{e:example}
 \end{align}
By Lemma~\ref{l:tdi}, the expression (\ref{e:example}) lies in $I$ as desired.

\section{Formal Variables} \label{s:formal}

We do not know how to establish {\em explicit} formulae such as (\ref{e:example})  to prove that
 \[ \det(D') - \det(D) \in I \] in general.  
Instead, we will prove {\em abstractly} the existence of polynomial identities such as (\ref{e:example}) that express the difference $\det(D') - \det(D)$ in terms of traces and determinants of elements of $\rho(\Delta)$.

For this, it is convenient to shift our perspective from working with the ring $\T$ to working with formal polynomial rings.  
We will define a polynomial algebra $R$ and a specialization homomorphism $\pi\colon R \longrightarrow K = \Frac(\T)$.  We will show that the 
identities we need hold in $R/\ker \pi$, so they apply in $\T$ as well by applying $\pi$. The advantage of working in $R$ is that we can identify the subring generated by traces and determinants of matrices mapping to $\rho(\Delta)$ under $\pi$ as the subspace of invariants under a certain group action, and use cohomological considerations to show that our element of interest lies in  this subspace.  We now describe this in greater detail.

\subsection{The ring \texorpdfstring{$R$}{R}} \label{s:trr}

Define
\[ R_0 = \Z[\boldsymbol{\epsilon}_{i}, \boldsymbol{\delta}_{ijk}],\]
the commutative polynomial ring in $r^2$ free variables enumerated by the entries of the matrix $D$.  Define
\[ R = R_0[\boldsymbol{a}_i, \boldsymbol{b}_i, \boldsymbol{c}_i, \boldsymbol{d}_i]_{i=1}^r, \]
a commutative polynomial ring in $r^2 + 4r$ free variables.  Define a ring homomorphism $R \lra K$ in the natural way indicated by our variable names, i.e.\ 
\[ \pi(\boldsymbol{\epsilon}_{i}) = \epsilon_i,\ \  \pi(\boldsymbol{\delta}_{ijk}) = \delta_{ijk}, \ \ \pi(\boldsymbol{a}_i) = a_i, \dotsc, \ \  \pi(\boldsymbol{d}_i) = d_i. \]
Note that $\pi(R_0) \subset \T$.

Finally, let $\boldsymbol{D}, \boldsymbol{D}' \in M_{r}(R)$ denote the natural matrices whose images under $\pi$ are equal $D, D'$, respectively, i.e. 
they are defined by making each entry bold.
Our goal is to prove that \[ \pi(\det(\boldsymbol{D}') - \det(\boldsymbol{D})) = \det(D') - \det(D) \in I. \]

\subsection{Relation Ideal}

We now define the polynomial relations that allow us to reduce $\det(\boldsymbol{D}') - \det(\boldsymbol{D})$ to an expression involving traces and determinants, as in the example of \S\ref{s:example}.  
Define \[ \boldsymbol{\rho_i} = \mat{\boldsymbol{a}_i}{\boldsymbol{b}_i}{\boldsymbol{c}_i}{\boldsymbol{d}_i} \in M_2(R). \]
Let $J \subset R$ be the ideal generated by the following:
\begin{itemize}
\item The 4 coefficients of 
\begin{equation} 
\sum_{i=1}^r \boldsymbol{\epsilon}_i \boldsymbol{\rho}_i \label{e:b1}
\end{equation}
 for each row of $\epsilon$-type in $D$.
\item The 4 coefficients of 
\begin{equation} \label{e:b2}
\boldsymbol{\rho}_i \boldsymbol{\rho}_j - \sum_{k = 1}^r \boldsymbol{\delta}_{ijk} \boldsymbol{\rho}_k
\end{equation} for each row of $\delta$-type in $D$.
\end{itemize}

It is clear that $J \subset \ker(\pi)$.

\subsection{Subring of traces and determinants} \label{s:subring}

Let $A \subset R$ denote the $R_0$-subalgebra
generated by the traces and determinants of all matrices in the noncommutative $\Z$-algebra generated by the matrices $\boldsymbol{\rho}_i$.
Denote by $\overline{A}$ the image of $A$ in $R/J$. We will show that in order to deduce our desired result  $\pi(\det(\boldsymbol{D}') - \det(\boldsymbol{D})) \in I$, 
it suffices to prove that the image of $\det(\boldsymbol{D}') - \det(\boldsymbol{D})$ in $R/J$ lies in $\overline{A}$.  
For this, it is important that $\det(\boldsymbol{D}') - \det(\boldsymbol{D})$ lies in the following ideal of $R$:
\[ I_R = \langle \boldsymbol{a}_i,   \boldsymbol{b}_i,  \boldsymbol{c}_i,  \boldsymbol{d}_i   \rangle. \]
\begin{lemma} \label{l:dir}
We have $\det(\boldsymbol{D}') - \det(\boldsymbol{D}) \in I_R.$
\end{lemma}
\begin{proof} This follows immediately from multilinearity of the determinant since every entry of 
$\boldsymbol{D}' - \boldsymbol{D}$ lies in $I_R$.
\end{proof}

\begin{lemma} We have $\pi(A \cap I_R) \subset I$. 
\label{l:air}
\end{lemma}

\begin{proof} Let $f \in \Z\langle X_1, \dotsc, X_r \rangle$ be an  element of the free polynomial algebra in $r$ noncommuting variables.
If $f$ has no constant term, then 
\[ \pi(f(\boldsymbol{\rho}_1, \dotsc,  \boldsymbol{\rho}_r)) = 
f(\rho_1, \dotsc, \rho_r) \in \rho(\Delta), \]
 and hence the trace and determinant of this element lies in $I$ by Lemma~\ref{l:tdi}.
The element $f(\boldsymbol{\rho}_1, \dotsc,  \boldsymbol{\rho}_r)$ clearly has coefficients lying in $I_R$. From these considerations, to prove the lemma it suffices to prove that
$\pi(R_0 \cap I_R) \subset I$.
  In fact, it is clear that $R_0 \cap I_R = 0$, as
  \[ R / I_R \cong R_0, \] 
  with $\boldsymbol{a}_i, \boldsymbol{b}_i, \boldsymbol{c}_i, \boldsymbol{d}_i \mapsto
  0$.    This concludes the proof.
\end{proof}

We summarize the result of this section: in order to prove the desired result 
 \begin{equation} \label{e:ddi}
 \det(D') - \det(D) \in I, 
 \end{equation}
  it suffices to prove that the image of 
$\det(\boldsymbol{D}') - \det(\boldsymbol{D})$ in $R/J$ lies in the subring $\overline{A}$, the image of $A$ in $R/J$.
Indeed, if this condition holds, then then \begin{equation} \label{e:daj}
\det(\boldsymbol{D}') - \det(\boldsymbol{D}) = a + j \end{equation} for some $a \in A, j \in J$.  Since $J \subset I_R$, we have
$a \in I_R$ by Lemma~\ref{l:dir} and hence $\pi(a) \in I$ by Lemma~\ref{l:air}.  Since $\pi(j) = 0$, the desired result (\ref{e:ddi}) follows from
applying $\pi$ to (\ref{e:daj}).

\section{Matrix Invariant Theory and Rational Cohomology}

The advantage of passing to the ring of formal variables $R$ (rather than working directly with $\T$ and $K$) is that we may identify the subring $A \subset R$ as the subspace invariant under a group action, rather than needing to write down explicit formulae.  To this end, we have the following important classical result, called the fundamental theorem of matrix invariant theory.

\begin{theorem}\label{t:invariants}
Endow the ring $\Z[\boldsymbol{a}_i, \boldsymbol{b}_i, \boldsymbol{c}_i, \boldsymbol{d}_i]_{i=1}^{r}$, with an action of $\GL_2(\Z)$ defined by simultaneous conjugation on the matrices $\boldsymbol{\rho}_i = \mat{\boldsymbol{a}_i}{\boldsymbol{b}_i}{\boldsymbol{c}_i}{\boldsymbol{d}_i}$. 
The subring of invariant elements is generated over $\Z$ by the traces and determinants of all matrices in the noncommutative algebra generated by the 
$\boldsymbol{\rho}_i$.
\end{theorem}

With notation as in \S\ref{s:trr}, we have the following corollary.

\begin{corollary}  \label{c:inv} Endow the ring $R$  with an action of $\GL_2(\Z)$ defined by simultaneous conjugation on the matrices $\rho_i = \mat{\boldsymbol{a}_i}{\boldsymbol{b}_i}{\boldsymbol{c}_i}{\boldsymbol{d}_i}$.  The subring of invariant elements is equal to $A$.
\end{corollary}

\begin{proof}  We may write $R = \Z[\boldsymbol{a}_i, \boldsymbol{b}_i, \boldsymbol{c}_i, \boldsymbol{d}_i]_{i=1}^{r} \otimes R_0$, where $\GL_2(\Z)$ acts trivially on $R_0$.  Since $R_0$ is $\Z$-flat, the result follows immediately from Theorem~\ref{t:invariants}.
\end{proof}

In our computations, it will be convenient to work with the Borel subgroup of $\GL_2$ consisting of lower triangular matrices.  We would like restriction to the Borel to induce an isomorphism on cohomology.  This is not true in general for ordinary group cohomology, so for this reason, we must work with the cohomology of algebraic groups called {\em rational cohomology}.  Here ``rational'' refers to actions defined by rational functions, rather than the rational numbers; throughout, we work integrally over $\Z$.

\subsection{Rational cohomology}

For simplicity, we will describe algebraic groups and their cohomology through their ``functor of points'' rather than via group schemes.
We therefore view the algebraic group $\G = \GL_2$ as the functor that associates to any commutative ring $S$ the group $\GL_2(S)$.

\begin{definition}  A {\em rational $\G$-module} is a $\Z$-module $V$ endowed with a functorial action, for every ring $S$, of the group $\GL_2(S)$ on $V \otimes S$, such that the action of a matrix $\mat{a}{b}{c}{d}$ is given by rational functions in the variables $a, b, c, d$.
  We say that $V$ is a rational $\G$-module over a commutative ring $R_0$ if $V$ has a structure of $R_0$-module that commutes with the $\G$-action.
\end{definition}

\begin{example} \label{e:adjoint}
 We denote by $\cA = \Z A \oplus \Z B \oplus \Z C \oplus \Z D$ the $\G$-module given by the adjoint representation, i.e. if $g = \mat{a}{b}{c}{d}$, then
\[ \mat{g \cdot A}{g \cdot B}{g \cdot C}{g \cdot D} = \mat{a}{b}{c}{d}^{-1} \mat{A}{B}{C}{D}  \mat{a}{b}{c}{d}. \]

Let $\B \subset \G$ denote the algebraic subgroup of lower triangular matrices.
Note that for $g = \mat{x}{0}{y}{z}$ in $\B$, we have \[ g \cdot A = A + \frac{y}{x} B, \qquad g \cdot B = \frac{z}{x} B, \qquad g \cdot D = D  -\frac{y}{x} B. \] In particular, $ \Z B$ is a  $\B$-submodule of $\cA$.
\end{example}

\begin{definition}[\cite{jantzen} Ch. 4]  If $V$ is a $\G$-module, define the invariants
\begin{align*} H^0(\G, V) = V^{\G} = \{ v \in V \colon \text{for all commutative rings $A$ and $g \in \G(A)$, } g \cdot v = v\}.  \end{align*}
The rational cohomology groups $H^i(\G, -)$ are the right derived functors of $H^0(\G, -)$.
\end{definition}

The same definitions hold with $\G$ replaced by $\B$.
The following is our motivation for using rational cohomology. 

\begin{theorem} \label{t:bg} Let $V$ be a $\G$-module.  The restriction map \[ H^i(\G, V) \rightarrow H^i(\B, V) \] is an isomorphism for all $i \ge 0$.
\end{theorem}

\begin{proof} See \cite{jjm}*{Theorem 4.4} for the general case.  In this paper we only need the case that $i=0$ and $V$ is a finitely generated $\Z$-module, so we describe the proof in this case. Injectivity is clear, so we need only prove surjectivity.
 Let $\boldsymbol{V}$ denote the spectrum of the tensor algebra of $V^* = \Hom(V, \Z)$.  This is the affine scheme whose points over a ring $S$ are $\boldsymbol{V}(S) = V \otimes S$.

Let $x \in H^0(\B, V)$.  If $S$ is any ring, we can define a map $\G(S) \lra \boldsymbol{V}(S) = V \otimes S$ by $g \mapsto g \cdot x$.  Since $x \in H^0(\B, V)$, this map factors through $G(S)/B(S) = \PP^1(S)$.  This yields a morphism of schemes $\PP^1 \lra \boldsymbol{V}$.  Since  $\boldsymbol{V}$ is affine, this morphism must be constant, i.e.\ 
$x \in H^0(\G, V)$.
\end{proof}

\begin{corollary} \label{c:a} We have $H^0(\B, R) = A$.
\end{corollary}

\begin{proof}  First note that $H^0(\G, R) = A$. Indeed, the inclusion $H^0(\G, R) \supset A$ is clear. Meanwhile the inclusion $H^0(\G, R) \subset A$ follows from Corollary~\ref{c:inv} and the observation that invariance under $\G$ is stronger than invariance under its group of points $\G(\Z)$.  The corollary then follows from Theorem~\ref{t:bg}.
\end{proof}

We conclude this section by stating a key input we will need from the cohomology theory of algebraic groups. See \cite{jjm}*{\S4.2} for a proof.

\begin{theorem}\label{t:adj} Let $\cA$ be the adjoint representation defined in Example~\ref{e:adjoint}.  For any nonnegative integer $k$, the $\B$-module $\cA^{\otimes k} \otimes R$ is acyclic,
i.e.\ $H^i(\B, \cA^{\otimes k} \otimes R) =0$ for $i \ge 1$.
\end{theorem}

\subsection{Roadmap} \label{s:roadmap}

Let $\boldsymbol{e} = \det(\boldsymbol{D}') - \det(\boldsymbol{D})$ and denote by
$\overline{\boldsymbol{e}}$ its image in $R/J$. Our strategy to prove that $\overline{\boldsymbol{e}}$ lies in $\overline{A}$ is as follows.  We will first prove that the ideal $J$ is stable under the action of $\B$.
We will then show that 
\[ \overline{\boldsymbol{e}} \in H^0(\B, R/J). \]
The long exact sequence in rational cohomology associated to 
\[ 0 \longrightarrow J \longrightarrow R \longrightarrow R/J \longrightarrow 0 \]
 yields an exact sequence
\begin{equation} \label{e:les}
 \begin{tikzcd}
H^0(\B, R) = A \ar[r] & H^0(\B, R/J) \ar[r, "c_J"] & H^1(\B, J).  \end{tikzcd}
\end{equation} 
The equality on the left is Corollary~\ref{c:a}.  Let $\beta \in H^1(\B, J)$ denote the image of $\overline{\boldsymbol{e}}$ under the connecting homomorphism $c_J$.
In view of (\ref{e:les}), the desired result $\overline{\boldsymbol{e}} \in \overline{A}$ will follow if we can show that $\beta =0$.

For this, we will define a certain $\B$-stable subideal $J' \subset J$ and show that in fact 
\[ \overline{\boldsymbol{e}}\in H^0(\B, R/J'). \]
(This is a slight abuse of notation, as here $\overline{\boldsymbol{e}}$ denotes the reduction of $\boldsymbol{e}$ modulo $J'$.)

We let $\alpha = c_{J'}(\overline{\boldsymbol{e}}) \in H^1(\B, J')$ defined as above.  Then $\beta = \iota_*(\alpha)$ where $\iota \colon J' \rightarrow J$ is the inclusion and $\iota_*$ is the induced map on rational cohomology.  To conclude, we will prove that the map
\[ \iota_*\colon H^1(\B, J') \longrightarrow H^1(\B, J) \]
vanishes.  Therefore $\beta = 0$, and our result follows.  

\subsection{Invariance}

Let $J' \subset J$ denote the subideal generated by the ``$b$'' coefficients of the matrices in (\ref{e:b1})--(\ref{e:b2}).  To be precise, $J'$ is generated by:
\begin{itemize}
\item The elements $
\sum_{i=1}^r \boldsymbol{\epsilon}_i \boldsymbol{b}_i 
$
 for each row of $\epsilon$-type in $D$.
\item The elements 
$
\boldsymbol{a}_i\boldsymbol{b}_j 
+ \boldsymbol{b}_i \boldsymbol{d}_j - \sum_{k = 1}^r \boldsymbol{\delta}_{ijk} \boldsymbol{b}_k
$ for each row of $\delta$-type in $D$.
\end{itemize}

\begin{lemma} \label{l:stable}
The ideals $J'$ and $J$ are stable under the action of $\B$.  More precisely,  the $\Z$-module spanned by each set of 4 relations in (\ref{e:b1})--(\ref{e:b2}) is isomorphic as a $\B$-module to a copy of the adjoint representation $\cA$ (see Example~\ref{e:adjoint}).
\end{lemma}

\begin{proof} The relations defining $J$ are linear combinations of products of the $\boldsymbol{\rho}_i$, with coefficients in $R_0$ (on which $\B$ acts trivially), and the definition of our action is by simultaneous conjugation on  $\boldsymbol{\rho}_i$. The second sentence of the Lemma follows.
The stability of $J'$ under $\B$ follows since the module of upper right entries $\Z B \subset \cA$ is stable under $\B$  (see Example~\ref{e:adjoint}).
\end{proof}

The goal of the rest of this subsection is to prove the following.

\begin{lemma} We have $\overline{\boldsymbol{e}} \in H^0(\B, R/J')$. \label{l:ebar}
\end{lemma}

\begin{proof}
 The matrix $\boldsymbol{D}$ has coefficients in $R_0$, on which $\B$ acts trivially, so we must show that 
\[ \overline{\det(\boldsymbol{D}')} \in H^0(\B, R/J'). \]
The group $\B$ is generated by elements of the form  $\sigma_{x,y} = \mat{x}{0}{0}{y}$ and $\tau_x = \mat{1}{0}{x}{1}$.  The matrix $\boldsymbol{D}'$ has coefficients in $R_0[\boldsymbol{a}_i, \boldsymbol{d}_i]$ on which  $\sigma_{x,y}$ acts trivially.  Therefore we must only consider the action of $\tau_x$.

The difference $\det(\tau_x(\boldsymbol{D}')) - \det(\boldsymbol{D}')$ is a linear combination (with coefficients $\pm 1$) 
of determinants of all matrices $M$ obtained by starting with $\boldsymbol{D}'$ and replacing some nonempty subset of the rows $w$ with $\tau_x(w) - w$.
We will show $\det(M) \in J'$ for each such matrix $M$.

Rows of $\epsilon$-type in $\boldsymbol{D}'$ contain only the elements $\boldsymbol{\epsilon}_i \in R_0$ that are fixed by the action of $\B$.  
So any matrix $M$ that contains a row $\tau_x(w) - w = 0$ for a row $w$ of $\epsilon$-type will have determinant 0.
Therefore we need only consider matrices $M$ that contain a row  $\tau_x(w) - w$ for a row $w$ of $\delta$-type.
Conjugation by $\tau_x$ sends $\boldsymbol{a}_i \mapsto \boldsymbol{a}_i + \boldsymbol{b}_i x$ and $\boldsymbol{d}_i \mapsto \boldsymbol{d}_i - \boldsymbol{b}_i x$.  Say $w$  is associated to a pair $(i,j)$.  Then 
\[ \tau_x(w) - w = (0, 0, \dots, \boldsymbol{b}_j x, 0, \dots, -\boldsymbol{b}_i x, 0, \dots, 0), \]
where there is $\boldsymbol{b}_j x$ in the $i$th slot and $ -\boldsymbol{b}_i x$ in the 
$j$th slot.
Note that if $i = j$ then $\tau_x(w) - w = 0$ and hence $\det(M)=0$.  So we assume $i \neq j$.
We make the following alterations to $M$ which do not change the determinant:
\begin{itemize}
\item We replace $\tau_x(w) - w$ by 
\[ (0, 0, \dots,  x, 0, \dots, -x, 0, \dots, 0). \]
At the same time, in every row other than $w$ we multiply the $j$th coordinate by $\boldsymbol{b}_j$ and the $i$th coordinate by $\boldsymbol{b}_i$. 
\item We then add the new  $j$th column to the new $i$th column.
\item For each $1 \le k \le r, k \neq i, j$, we add $\boldsymbol{b}_k$ times the $k$th column to the 
new $i$th column.
\end{itemize}
In the matrix $M$ that results from these changes, the $i$th coordinate is precisely the generator of $J'$ associated to the row.
Therefore $M$ has an entire column with elements lying in $J'$, so its determinant lies in $J'$ as well.
\end{proof}

\subsection{Kozsul complex}

For simplicity, denote the $r$ elements of $J'$ corresponding to the rows of $\boldsymbol{D}'$ by $B_1, \dots, B_r$, i.e.
\[ \boldsymbol{D}' \begin{pmatrix} \boldsymbol{b}_1 \\ \vdots \\ \boldsymbol{b}_r \end{pmatrix}  = 
\begin{pmatrix} B_1 \\ \vdots \\ B_r \end{pmatrix}. \]

  The rank 1 modules $\Z B_i $ are rational $\B$-representations isomorphic to $\Z B \subset \cA$ as in Example~\ref{e:adjoint}.  Let $V = \oplus_{i=1}^r \Z B_i$.  Consider the following complex of rational $\B$-representations over $R$:
\begin{equation} \label{e:kozsul} \begin{tikzcd}
 0 \ar[r] & \left(\bigwedge\nolimits^r V\right) \otimes R \ar[r,"f_r"] 
 & \left(\bigwedge\nolimits^{r-1} V \right) \otimes R   \ar[r, "f_{r-1}"]
 &  \cdots  \left(\bigwedge\nolimits^2 V\right) \otimes R \ar[r,"f_2"] & V \otimes R \ar[r,"f_1"] &  R. 
 \end{tikzcd}
 \end{equation}
 Here all wedge products and tensor products are over $\Z$.
The maps  \[ f_i\colon \left(\bigwedge\nolimits^i V\right) \otimes R \rightarrow \left(\bigwedge\nolimits^{i-1} V\right) \otimes R \]are given  by 
\[ B_{k_1} \wedge B_{k_2} \wedge \cdots \wedge B_{k_i} \otimes r \mapsto
\sum_{j=1}^i (-1)^j B_{k_1} \wedge B_{k_2} \wedge \cdots \wedge \hat{B}_{k_j} \wedge \cdots B_{k_i} \otimes B_{k_j}r. \]
Noting that each term  $\left(\bigwedge\nolimits^r V\right) \otimes R$ in (\ref{e:kozsul}) is isomorphic to 
$\bigwedge \nolimits^r_R \left(V \otimes R\right)$, the sequence (\ref{e:kozsul}) is precisely the Kozsul complex for the free $R$-module $V \otimes R$.  It is therefore exact if we can prove that the elements $B_1, \dots, B_r$ form a regular sequence in $R$.

\begin{lemma} The elements $B_1, \dots, B_r$ form a regular sequence in $R$. 
\end{lemma}

After enacting the change of variables for rows of $\delta$-type:
\[ \delta_{ijk}' = \begin{cases} \delta_{iji} - d_j & k=i \\
\delta_{ijj} - a_i & k=j \\
\delta_{ijk} & k \neq i,j,
\end{cases} \]
we see the lemma follows from the following.

\begin{lemma} \label{l:regular}
Let $S$ be a commutative ring and let $T = S[x_{ij}]_{i,j=1}^{r}$.  The elements $L_i = \sum_{j=1}^r x_{ij} y_j$ form a regular sequence in $T[y_1, \dots, y_r]$.  
\end{lemma}

We leave the proof of Lemma~\ref{l:regular} as an exercise, referring the reader to \cite{jjm}*{Proposition 5.13}.

\subsection{Embedding into an acyclic complex}

 \begin{theorem}\label{t:comm}
There is a commutative diagram of complexes of rational $\B$-modules
\begin{equation} \label{e:kozsul2} \begin{tikzcd}
0 \ar[r] \ar[d] &[-1em] \left(\bigwedge\nolimits^r V\right) \otimes R \ar[r,"f_r"] \ar[d,"\iota_r"]
 & \left(\bigwedge\nolimits^{r-1} V \right) \otimes R   \ar[r, "f_{r-1}"] \ar[d,"\iota_{r-1}"]
 &  \cdots  \left(\bigwedge\nolimits^2 V\right) \otimes R \ar[r,"f_2"] \ar[d, "\iota_2"] & V \otimes R \ar[r,"f_1"]  \ar[d,"\iota_1"] & J' \ar[d, "\iota_0"] \\
 0 \ar[r] & W_r \otimes R \ar[r,"g_r"] 
 & W_{r-2} \otimes R   \ar[r, "g_{r-1}"]
 &  \cdots  W_2 \otimes R \ar[r,"g_2"] & W_1 \otimes R \ar[r,"g_1"] &  J,
 \end{tikzcd}
  \end{equation}
where the $W_i \otimes R$ are acyclic $\B$-modules for $i \ge 1$.
 \end{theorem}
 
 \begin{proof}
 We have already constructed the complex in the top row (see (\ref{e:kozsul})), noting that the image of $f_1$ is the ideal generated by the $B_i$, namely $J'$.
  Next we define the bottom row. Fix an index $i=1, \dots, r$, and suppose that the $i$th row of $D$ corresponds to a pair $(m,n)$, i.e.\ we have $B_i = B_{mn}$ where
   \[ \boldsymbol{\rho}_m \boldsymbol{\rho}_n - \sum_{k=1}^r \boldsymbol{\delta}_{mnk} \boldsymbol{\rho}_k = \mat{A_{mn}}{B_{mn}}{C_{mn}}{D_{mn}}.
  \]
We then let \[  \cA_i = \Z A_{mn} \oplus \Z B_{mn} \oplus \Z C_{mn} \oplus \Z D_{mn} \] denote the corresponding copy of the adjoint.  Define 
\[ W_i =  \bigoplus_{ \{k_1, k_2, \dots, k_i\} \subset \{1, \dotsc, r\} }\cA_{k_1} \otimes \cdots \otimes \cA_{k_i}. \]

The vertical maps $\iota_i$ are given by, for $k_1 < k_2 < \cdots < k_i$,
\[ B_{k_1} \wedge \cdots \wedge B_{k_i} \otimes r \mapsto B_{k_1} \otimes \cdots \otimes B_{k_i} \otimes r. \]
The maps $g_i$ are given by:
\[ X_1 \otimes X_2 \otimes \cdots \otimes X_i \otimes r \mapsto \sum_{j=1}^{i} 
(-1)^jX_1 \otimes X_2 \otimes \cdots  \hat{X}_j \cdots \otimes X_i \otimes X_j r. \]
The image of the map $g_1$ is precisely our ideal $J$.  The fact that the bottom row of (\ref{e:kozsul2}) is a complex, as well as the commutativity of the diagram, is clear. The $\B$-acyclicity of the $W_i \otimes R$ follows from Theorem~\ref{t:adj}.
 \end{proof}
 
 \subsection{A cascade of cohomology classes}
 
 \begin{theorem} \label{t:push}
 Let $\iota \colon J' \rightarrow J$ be the inclusion and let  
\[ \iota_*\colon H^1(\B, J') \longrightarrow H^1(\B, J) \]
be the induced map on first rational cohomology groups.  Then $\iota_* = 0$.
 \end{theorem}
 
 \begin{proof} Let the notation be as in Theorem \ref{t:comm}.
Recall $\im(f_1) = J'$. Let  
 \[ \alpha_1  \in H^1(\B, J') = H^1(\B, \im(f_1)). \]
We need to show that $\iota_{0, *} \alpha_1 = 0$.
The long exact sequence in cohomology associated to 
\[ 0 \rightarrow \ker(f_1) \rightarrow V \otimes R \rightarrow \im(f_1) \rightarrow 0 \]
 yields a class $\alpha_2 \in H^2(\B, \ker(f_1))$ that represents the obstruction to lifting $\alpha_1$ to a class in $H^1(\B, V \otimes R)$.  Writing $\ker(f_1) = \im(f_2)$, we can view $\alpha_2 \in H^2(\B, \im(f_2))$ and repeat the process above, using the coboundary in the long exact sequence associated to 
\[ 0 \rightarrow \ker(f_2) \rightarrow \bigwedge\nolimits^2 V \otimes R \rightarrow \im(f_2) \rightarrow 0 \]
to obtain $\alpha_3 \in H^3(\B, \ker(f_2))$.  Continuing in this way we obtain 
\[ \alpha_i \in H^i(\B, \ker(f_{i-1})) = H^i(\B, \im(f_i)) \] for $i = 1, \dots, r+1$.  Note that $\alpha_{r+1} = 0$ since $\ker(f_r)=0$.

   For each $i = 1, \dotsc, r+1$,  define
\[ \beta_i = \iota_{i-1, *} \alpha_i \in H^i(\B, \im(g_{i})). \]  In particular, we have $\beta_1 = \iota_{0, *} \alpha_1$, which is the class we are trying to show vanishes.  The bottom row of (\ref{e:kozsul2}) is a complex but we do not claim it is exact.  Nevertheless the obstruction to $\beta_i \in H^i(\B, \im(g_{i}))$  lifting to a class in $H^i(\B, W_i \otimes R)$ is precisely the image of $\beta_{i+1} \in H^{i+1}(\B, \im(g_{i+1}))$ in $H^{i+1}(\B, \ker(g_i))$. Now $\beta_{r+1} = 0$ since $\alpha_{r+1}=0$, and hence we conclude that $\beta_r$ lifts to a class in $H^r(\B, W_r \otimes R)$.  However, by acyclicity we have $H^r(\B, W_r \otimes R) =0$ and hence $\beta_r=0$.  Therefore, $\beta_{r-1}$ lifts to a class in  $H^{r-1}(\B, W_{r-1} \otimes R)$; again this cohomology group vanishes so $\beta_{r-1} = 0$.  This downward cascading continues and we obtain $\beta_i = 0 $ for all $i$.  In particular $\beta_1 = 0$ as desired.
 \end{proof}

As explained in \S\ref{s:roadmap}, Theorem~\ref{t:push} implies that $\overline{\boldsymbol{e}} \in H^0(\B, R/J)$ lies in $\overline{A} \subset R/J$.  
From the discussion of \S\ref{s:subring}, it follows that $\det(D') - \det(D) = - \det(D) \in I$.  This completes the proof of Theorem~\ref{t:v3}.

\end{document}